\documentclass[10pt]{amsart}
\usepackage{amsmath}
\usepackage[usenames,dvipsnames]{color}
\usepackage{parskip}
\usepackage{amsfonts}
\usepackage{amscd}
\usepackage[centertags]{amsmath}
\usepackage{amssymb}
\usepackage[all,cmtip]{xy}
\usepackage[english]{babel}
\usepackage{tabularx}
\usepackage{amsxtra}
\usepackage{euscript}
\usepackage[T1]{fontenc}
\usepackage{doc, exscale, fontenc, latexsym, syntonly}
\usepackage{amsfonts}
\usepackage{amsthm}
\usepackage{graphicx}
\usepackage{mciteplus}
\usepackage{cite}

\newtheorem{theorem}{Theorem}[section]
\newtheorem{corollary}[theorem]{Corollary}
\newtheorem{lemma}[theorem]{Lemma}
\newtheorem{proposition}[theorem]{Proposition}
\newtheorem{remark}[theorem]{Remark}
\theoremstyle{definition}
\newtheorem{definition}[theorem]{Definition}

\numberwithin{figure}{section}
\numberwithin{table}{section}

\newcommand*\acknowledgment[1]{%
	\begingroup\noindent
	\rightskip\leftskip
	\begin{flushleft}\textbf{\large Acknowledgment.}\, #1%
		\par\vspace*{1mm}\end{flushleft}\endgroup}
\begin{document}

\title[Topological Complexities of Finite Digital Images]{Topological Complexities of Finite Digital Images}

\author{MEL\.{I}H \.{I}S and \.{I}SMET KARACA}
\date{\today}

\address{\textsc{Melih Is}
Ege University\\
Faculty of Sciences\\
Department of Mathematics\\
Izmir, Turkey}
\email{melih.is@ege.edu.tr}
\address{\textsc{Ismet Karaca}
Ege University\\
Faculty of Science\\
Department of Mathematics\\
Izmir, Turkey}
\email{ismet.karaca@ege.edu.tr}



\begin{abstract}
Digital topological methods are often used on computing the topological complexity of digital images. We give new results on the relation between reducibility and digital contractibility in order to determine the topological complexity of a digitally connected finite digital image. We present all possible cases of the topological complexity TC of a finite digital image in $\mathbb{Z}$ and $\mathbb{Z}^{2}$. Finally, we determine the higher topological complexity TC$_{n}$ of finite irreducible digital images independently of the number of points for $n > 1$.
\end{abstract}

\maketitle

\section{Introduction}
\label{intro}

\quad One of the main streams of topological robotics is to apply topological ideas to solve specific problems of engineering and computer science. On the other hand, digital topology has an important place in the studies of computer science. Topological robotics and digital topology have a common field of study and common methods. This raises the question: What results can one get in the subject of robotics by using topological methods on digital images? The answer gets inspired with the study of using discrete structures on computing topological complexity numbers.

\quad Studies of topological robotics start with defining the notion of the topological complexity number of a path-connected topological space by Farber \cite{Farber:2003}. This number is an integer that indicates the complexity of area where the robot moves. Many different methods, especially cohomology, are used in algebraic topology to determine the number exactly (see \cite{Farber:2008} for a collection of the methods used). Contractibility of a topological space is so important if one wants to know the topological complexity number precisely. The topological complexity number of a contractible space is $1$. If a topological complexity number of a topological space is $1$, then the space must be contractible \cite{Farber:2003}. Rudyak \cite{Rudyak:2010} improves the idea of this topological complexity definition and presents the higher topological complexity number of a topological space. He proves that the special version of this new number corresponds to Farber's topological complexity number. Karaca and Is \cite{KaracaIs:2018} defines the digital topological complexity number and the digital higher topological complexity number \cite{MelihKaraca} by moving the study to the field of digital topology. Digital topology is a discrete structure built on digital images at the point, so it assembles topological features without including a topology (see \cite{BorVer:2018,Boxer:1999,Boxer:2005,Boxer:2006,Boxer2:2006,BoxKar:2012,BoxStaecker:2020,Boxer:2020,Chen:2004,ChenRong:2010,ÇınarEgeKaraca:2020,EgeKaraca:2013}, \cite{Khal:1987}, and \cite{MorgenRosen:1981,Ros:1970,Ros:1979} for more information about digital topology, its some applications and digital geometry). This fundamental difference makes some of the topological methods useless in digital topology. For instance, cohomological cup-product method is one of the well-known methods in usual topology to have a new bound for the topological complexity number \cite{Farber:2003}. But it does not work for digital images \cite{MelihKaraca}. At this point, it is sometimes necessary to use new ways that comply with the rules of the digital topology. It is not only a problem of studies of digital topological complexity but also a problem of studies in every aspect of digital topology. As an example, the Euler characteristic is not a homotopy invariant for digital images \cite{EgeKaracaMeltem:2014}. Staecker et al. \cite{HaarMurphyPetStaecker:2015} have a new numerical homotopy invariant for digitally connected digital images and regard their invariant as 'true', which means that it is not an adaptation from topology. They use the notions of reducibility and rigidity. In this paper, we examine a relation between digital contractibility and reducibility (partly rigidity). This leads to us to have a characterization of finite digital images in $\mathbb{Z}$ and $\mathbb{Z}^{2}$ in terms of the topological complexity and the higher topological complexity.

\quad First, we have a simple background of digital setting and recall the definitions of the topological complexity and the higher topological complexity with some important properties. Later, we show that if $X$ is an irreducible digital image, then the topological complexity of the image is greater than $1$. We also demonstrate under what conditions the reducibility requires the digital contractibility. We prove that if $X \subset \mathbb{Z}$ is a digitally connected finite image, then the topological complexity of the image is $1$. After that, we examine the topological complexity of irreducible images having finite number of points. Using this fact, we have the topological complexity number of all digitally connected finite digital images in $\mathbb{Z}^{2}$. We conclude that there is no digitally connected finite image in $\mathbb{Z}$ and $\mathbb{Z}^{2}$ such that the topological complexity of the image is greater than $2$. In Section \ref{sec:2}, we consider the diagonal map on a digital image $X$ and define a new digital fibrational substitute of it. Then we find the digital higher topological complexity number of irreducible images with computing the digital Schwarz genus of the digital fibrational substitute. The topological complexity of the irreducible images is independent from the number of points. At the end of the paper, we state some open problems.

\section{Preliminaries}
\label{sec:1}
\quad This section is planned to provide some backgrounds commonly used in digital topology and topological robotics.

\quad A digital image is the basic element of the digital topology and consists of a set with a relation on this set. More precisely, $(X,\kappa)$ is a \textit{digital image} \cite{Boxer:1999}, where $X$ is a finite subset of $\mathbb{Z}^{n}$ and $\kappa$ is an adjacency relation for the points of $X$. On a digital image, unlike in topological spaces, there is an adjacency relation instead of topology and this relation works as follows: Let $X$ be a finite subset of $\mathbb{Z}^{n}$ and let $k \in \mathbb{Z}$ with $1 \leq k \leq n$. For any distinct elements $x = (x_{1}, ..., x_{n})$, $y = (y_{1}, ..., y_{n}) \in X$, $x$ and $y$ are called \textit{$c_{k}-$adjacent} \cite{Boxer:1999} if we have $|x_{i} - y_{i}| = 1$ for at most $k$ indices $i$, and $|x_{j} - y_{j}| \neq 1$ implies $x_{j} = y_{j}$ for all indices $j$. The notation $x \leftrightarrow_{c_{k}} y$ is used when $x$ is adjacent to $y$. By this construction, we have $c_{1} = 2$ adjacency in $\mathbb{Z}$, $c_{1} = 4$ and $c_{2} = 8$ adjacencies in $\mathbb{Z}^{2}$, and $c_{1} = 6$, $c_{2} = 18$ and $c_{3} = 26$ adjacencies in $\mathbb{Z}^{3}$. Let $(X,\kappa)$ and $(Y,\lambda)$ be any digital images. Let $(x_{1},y_{1})$ and $(x_{2},y_{2})$ be any two points in the cartesian product image $X \times Y$. Then \textit{$(x_{1},y_{1})$ and $(x_{2},y_{2})$ are adjacent in $X \times Y$} \cite{BoxKar:2012} if one of the following conditions holds:
\begin{itemize}
	\item $x_{1} = x_{2}$ and $y_{1} = y_{2}$; or
	
	\item $x_{1} = x_{2}$ and $y_{1} \leftrightarrow_{\lambda} y_{2}$; or
	
	\item $x_{1} \leftrightarrow_{\kappa} x_{2}$ and $y_{1} = y_{2}$; or	
	
	\item $x_{1} \leftrightarrow_{\kappa} x_{2}$ and $y_{1} \leftrightarrow_{\lambda} y_{2}$.
\end{itemize}

\quad Let $(X,\kappa)$ be a digital image in $\mathbb{Z}^{n}$ and let $p$ be any point in $X$. A \textit{$\kappa-$neighbor} \cite{Herman:1993} of $p$ is the point that is $\kappa-$adjacent to $p$. Let $(X,\kappa) \subset \mathbb{Z}^{n}$ be a digital image. $X$ is called \textit{$\kappa-$connected} \cite{Herman:1993} if and only if for every pair of different points $x$, $y \in X$, there is a set $\{x_{0},x_{1}, ...,x_{m}\}$ of points in $X$ such that $x=x_{0}$, $y=x_{m}$ and $x_{i} \leftrightarrow_{\kappa} x_{i+1}$ for $i = 0, 1, ..., m-1$. Let $f : (X_{1},\kappa_{1}) \rightarrow (X_{2},\kappa_{2})$ be a digital map such that $X_{1} \subset \mathbb{Z}^{m_{1}}$ and $X_{2} \subset \mathbb{Z}^{m_{2}}$. Then $f$ is said to be \textit{$(\kappa_{1},\kappa_{2})-$continuous} \cite{Boxer:1999} if $x \leftrightarrow_{\kappa_{1}} x^{'}$ for any different $x$, $x^{'} \in X_{1}$, then $f(x) \leftrightarrow_{\kappa_{2}} f(x^{'})$ in $X_{2}$. In addition, $f$ is \textit{$(\kappa_{1},\kappa_{2})-$isomorphism} \cite{Boxer2:2006} if $f$ is bijective, $(\kappa_{1},\kappa_{2})-$continuous and the inverse $f^{-1}$ is $(\kappa_{2},\kappa_{1})-$continuous.

\quad A set $[a,b]_{\mathbb{Z}} = \{z \in \mathbb{Z} : a \leq z \leq b\}$ is called a \textit{digital interval} \cite{Boxer:2006} from $a$ to $b$. Since the interval is a subset of $\mathbb{Z}$, it has $2-$adjacency. If a digital map $f:[0,m]_{\mathbb{Z}} \rightarrow X$ is  $(2,\kappa)-$continuous with $f(0)=x$ and $f(m)=y$, then $f$ is a \textit{digital path} \cite{Boxer:2006} from $x$ to $y$ in $X$. The digital path $f$ is called a $\kappa-$loop if $f(0) = f(m)$. The product of two digital paths defined in \cite{Khal:1987}: Let $f:[0,m]_{\mathbb{Z}} \rightarrow X$ and $g:[0,n]_{\mathbb{Z}} \rightarrow X$ be  digital $\kappa-$paths with $f(m) = g(0)$. Then \textit{the product of $f$ and $g$} is defined as the map
\[(f \ast g):[0,m+n]_{\mathbb{Z}} \rightarrow X\] 
\begin{displaymath} \hspace*{5.9cm} t \longmapsto (f \ast g)(t) = \begin{cases} f(t), & t \in [0,m]_{\mathbb{Z}} \\ g(t-m), & t \in [m,m+n]. \end{cases} \end{displaymath}

\quad Let $(X,\kappa)$ and $(Y,\lambda)$ be two  digital images, and let $f$, $g:X \rightarrow Y$ be any $(\kappa,\lambda)-$continuous maps. The maps $f$ and $g$ are \textit{$(\kappa,\lambda)-$homotopic} \cite{Boxer:1999} if there exists $m \in \mathbb{Z}$ such that for all $x \in X$, there is a digital map $F : X \times [0,m]_{\mathbb{Z}} \rightarrow Y$ with $F(x,0) = f(x)$ and $F(x,m) = g(x)$, for any fixed $t \in [0,m]_{\mathbb{Z}}$, the digital map $F_{t} : X \rightarrow Y$ is $(\kappa,\lambda)-$continuous and for any fixed $x \in X$, the digital map $F_{x}:[0,m]_{\mathbb{Z}} \rightarrow Y$ is
$(2,\lambda)-$continuous. It is denoted by $f \simeq_{(\kappa,\lambda)} g$ when $f$ is $(\kappa,\lambda)-$homotopic to $g$. We also note that $m$ is the step number of the homotopy in this construction. In other saying, we say that $f$ is digitally homotopic to $g$ in $m$ step.

\quad Let $f : X \rightarrow Y$ be a $(\kappa,\lambda)-$continuous map. Then $f$ is a \textit{$(\kappa,\lambda)-$homotopy equivalence} \cite{Boxer:2005} if there exists a $(\lambda,\kappa)-$continuous map $g : Y \rightarrow X$ for which $g \circ f$ is digitally homotopic to the identity function on $X$ and $f \circ g$ is digitally homotopic to the identity function on $Y$. A digital image $X$ is said to be \textit{$\kappa-$contractible} \cite{Boxer:1999} if the identity map on $X$ is $(\kappa,\kappa)-$homotopic to a constant map $c$ for some $x_{0} \in X$, where the constant map $c : X \longrightarrow X$ is defined by $c(x) = x_{0}$ for all $x \in X$.

\begin{definition}\cite{HaarMurphyPetStaecker:2015}
	Let $(X,\kappa)$ be a finite digital image. If $X$ is $(\kappa,\kappa)-$homotopy equivalent to an image of fewer points, then $X$ is called \textit{reducible}. If $X$ is not reducible, then $X$ is said to be \textit{irreducible}. 
\end{definition} 

\begin{definition}\cite{HaarMurphyPetStaecker:2015}
	Let $(X,\kappa)$ be a finite digital image. If the identity map on $X$ is the only map that is $(\kappa,\kappa)-$homotopic to the identity map on $X$, then $X$ is \textit{rigid}.  
\end{definition}

\quad Let $(X,\kappa)$ be a digital image. If there is an integer $m \geq 4$ for which there exists a $(2,\kappa)-$continuous map $f : [0,m-1]_{\mathbb{Z}} \rightarrow X$ such that the following conditions hold:
\begin{itemize}
	\item $f$ is bijective; 
	
	\item $f(0) \leftrightarrow_{\kappa} f(m-1)$; and
	
	\item for all $t \in [0,m-1]_{\mathbb{Z}}$, the only $\kappa-$neighbors of $f(t)$ in $f([0,m-1]_{\mathbb{Z}})$ are $f((t-1) mod \hspace*{0.2cm} m)$ and $f((t+1) mod \hspace*{0.2cm} m)$,
\end{itemize}
then $X$ is a \textit{digital simple closed $\kappa-$curve} \cite{Boxer:2005}. A simple closed curve with $m$ points is generally denoted by $C_{m}$ and named as an \textit{$m-$gon} or a \textit{digital $m-$cycle}. Let ($X,\kappa)$ be a digital image. An \textit{$m-$loop} \cite{HaarMurphyPetStaecker:2015} is a digitally continuous map from $C_{m}$ to $X$. Moreover, the map $p$ is called a \textit{simple $m-$loop} if $p$ is an injection with $p(c_{i}) \leftrightarrow_{\kappa} p(c_{i+1})$ in $X$ such that there are no other adjacencies between points in the image of $C_{m}$.

\begin{proposition}\cite{HaarMurphyPetStaecker:2015} \label{proposition 5}
	$C_{m}$ is irreducible for $m \geq 5$.
\end{proposition}

\begin{definition}\cite{HaarMurphyPetStaecker:2015} \label{definition 1}
	$L_{m}(X)$ is an integer which counts the number of equivalence classes of $m-$loops for any finite digital image $X$.
\end{definition}

\begin{theorem}\cite{HaarMurphyPetStaecker:2015} \label{theorem 3}
	Let $(X,\kappa)$ and $(Y,\lambda)$ be any two digital images such that they are digitally homotopy equivalent. Then for all positive integer $m$, we get $L_{m}(X) = L_{m}(Y)$. 
\end{theorem}

\quad The next three results are the basic facts that we often use in next sections. By using these results, we have an idea about the digital topological complexity of a finite digital image (reducible or irreducible) with respect to the number of points.

\begin{proposition}\cite{HaarMurphyPetStaecker:2015} \label{proposition 2}
	Let $(X,\kappa)$ be a finite digital image. If $X$ has no simple $m-$loop for any $m \geq 4$, then $X$ is digitally homotopy equivalent to a one-point digital image.
\end{proposition}

\begin{proposition}\cite{HaarMurphyPetStaecker:2015} \label{proposition 3}
	Let $(X,\kappa)$ be a digitally connected digital image having $m$ points. If $m \leq 4$, then $X$ is digitally homotopy equivalent to a one-point digital image.
\end{proposition}

\begin{proposition}\cite{HaarMurphyPetStaecker:2015} \label{proposition 4}
	Let $X$ be a digitally connected digital image having five points. Then $X$ is digitally homotopy equivalent to a one-point digital image or to $C_{5}$.
\end{proposition}

\quad Let $PX$ be a set of all digitally continuous digital paths for any $\kappa-$connected digital image $(X,\kappa)$. Let $s : X \times X \rightarrow PX$ be the digital map which takes any pair $(a,b)$ of a digital image to a digital path starting at $a$ and ending at $b$, is denoted by the digital version of motion planning algorithm. In \cite{KaracaIs:2018}, there is a reasoned way to define the continuity of motion planning algorithm. The digital connectedness on $PX$ is defined as follows: let $\tau$ be an adjacency relation on $PX$, and let $\alpha$ and $\beta$ be any digital paths on $X$. If $\alpha$ and $\beta$ are $\tau-$connected for all $t \in [0,m]_{\mathbb{Z}}$, then $\alpha \leftrightarrow_{\kappa} \beta$. $\alpha$ and $\beta$ can have different steps in their way. For instance, when $\alpha$ has $5$ steps and $\beta$ has $2$ steps, the last step of $\beta$ repeats itself $3$ times. Then both $\alpha$ and $\beta$ have the same number of steps, which means there is no confusion about the adjacency of digital paths. See \cite{KaracaIs:2018} for more detail and example about continuity of digital motion planning algorithm. Moreover, $\pi : PX \rightarrow X \times X$ is a digital map, which takes any digital path $\alpha$ to the pair $(\alpha(0),\alpha(m))$, where $\alpha(m)$ is the final step of $\alpha$. Finally, we are ready to give the definition:

\begin{definition}\cite{KaracaIs:2018}
	The \textit{digital topological complexity} TC$(X,\kappa)$ is the minimal number $k$ such that $$X \times X = U_{1} \cup U_{2} \cup ... \cup U_{k}$$ with the property that there exists a digitally continuous motion planning algorithm $s_{j} : U_{j} \rightarrow PX$, $j = 1, 2, ..., k$, for which $\pi \circ s_{j}$ is identity map over each $U_{j} \subset X \times X$. If no such $k$ exists, then TC$(X,\kappa) = \infty$.
\end{definition}

\quad We compute the digital topological complexity of only connected digital images (recall that in ordinary topology, only path-connected topological spaces are considered for the computation of the topological complexity). The next proposition is quite important such as the fact that the topological complexity is a homotopy invariant.

\begin{proposition}\cite{KaracaIs:2018} \label{proposition 1}
	TC$(X,\kappa) = 1$ if and only if $(X,\kappa)$ is $\kappa-$ contractible.
\end{proposition}

\begin{definition}\cite{MelihKaraca}
	Let $f:(X,\kappa) \rightarrow (Y,\lambda)$ be a map in digital images with digitally connected spaces $(X,\kappa)$ and $(Y,\lambda)$. A digital fibrational substitute of $f$ is defined as a digital fibration $\widehat{f}:(Z,\kappa_{3}) \longrightarrow (Y,\kappa_{2})$ such that there exists a commutative diagram
	
	\begin{displaymath}
	\xymatrix{
		X \ar[r]^{h} \ar[d]_f &
		Z \ar@{.>}[d]^{\widehat{f}} \\
		Y \ar@{=}[r]_{1_{Y}} & Y, }
	\end{displaymath}
	
	where $h$ is a digital homotopy equivalence.
\end{definition}

\quad Let $p : X \rightarrow Y$ be a digital fibration. \textit{The digital Schwarz genus} \cite{MelihKaraca} of $p$ is defined as the minimum number $k$ such that $X = U_{1} \cup U_{2} \cup ... \cup U_{k}$ with the property that for all $1 \leq i \leq k$, there is a digitally continuous map $s_{i} : U_{i} \rightarrow X$ that satisfies $p \circ s_{i} = id_{U_{i}}$. If we do not have a digital fibration, then we regard the digital Schwarz genus of a map as the digital Schwarz genus of its digital fibrational substitute. Consequently, we now give another important definition:

\begin{definition}\cite{MelihKaraca}
	Let $X$ be any $\kappa$-connected digital image. Let $J_{n}$ be the wedge of $n-$digital intervals $[0,m_{1}]_{\mathbb{Z}}, ... , [0,m_{n}]_{\mathbb{Z}}$ for a positive integer $n$, where $0_{i} \in [0,m_{i}]$, $i = 1, ... ,n$, are identified. Then \textit{the digital higher topological complexity} TC$_{n}(X,\kappa)$ is defined by the digital Schwarz genus of the digital fibration $$e_{n}:X^{J_{n}} \rightarrow X^{n}$$
	\hspace*{5.7cm} $f \longmapsto (f(m_{1})_{1},...,f(m_{n})_{n})$,\\
	where $(m_{i})_{k}$, $k=1, ..., n$ denotes the endpoints of the $i-$th interval for each $i$.
\end{definition}

\quad In the definition of the higher topological complexity in digital images, we have TC$_{2} =$ TC \cite{MelihKaraca}. Furthermore, TC$_{n}$ is also a homotopy invariant for digital images just as TC.

\section{Digital Topological Complexity in $\mathbb{Z}$ and $\mathbb{Z}^{2}$}
\label{sec:2}
\quad We begin with discussing the relation between the contractibility and the reducibility on digitally connected digital images. It is clear that if $(X,\kappa)$ is a $\kappa-$connected and $\kappa-$contractible finite digital image, then $X$ is reducible. The converse need not to be true. For example, consider the following digital image $X$ with $8-$adjacency and its digital homotopy equivalence in Figure \ref{fig1:figure1}:
\begin{figure*}[h]
	\centering
	\includegraphics[width=0.80\textwidth]{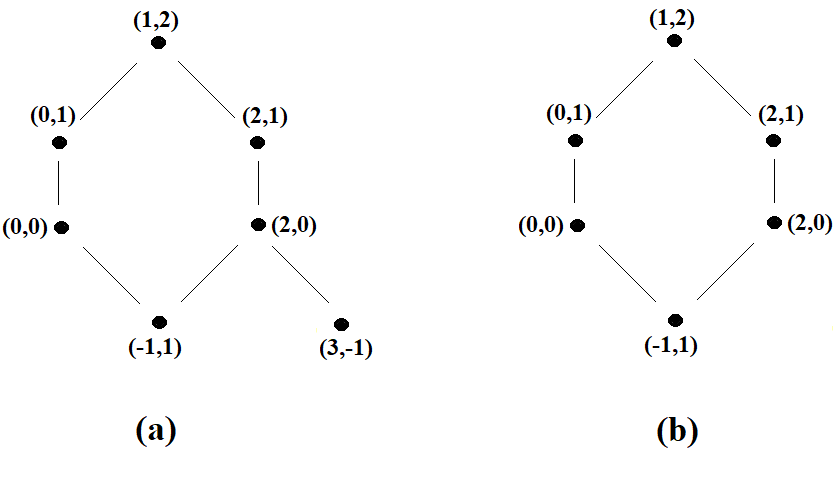}
	\caption{The digital image $X$ with $8-$adjacency is on the left (a) and its digital homotopy equivalence $X \setminus \{(3,-1)\}$ on the right (b).}
	\label{fig1:figure1}
\end{figure*}

The digital image $X$ is reducible because it is digitally homotopy equivalent to the image $X \setminus \{(3,-1)\}$ (Figure \ref{fig1:figure1} (b)) but it is well-known that $X$ is not $8-$contractible. Combining this result with Proposition \ref{proposition 1}, we have that the topological complexity number of a reducible image can be different from $1$. Indeed, we obtain that 
\begin{eqnarray*}
	\text{TC}(X,8) = \text{TC}(X \setminus \{(3,-1)\},8) = 2 
\end{eqnarray*} (see [Example 3.5, \cite{KaracaIs:2018}]). 
In addition, if for any digitally connected finite image $(X,\kappa)$ having more than one point, TC$(X,\kappa) = 1$, then $X$ must be reducible. So, we immediately have the result:

\begin{proposition} \label{corollary 1}
	Let $(X,\kappa)$ be a digitally connected finite image having more than one point. If $X$ is irreducible, then TC$(X,\kappa) > 1$. 
\end{proposition}  

\quad We note that Proposition \ref{corollary 1} is still true if we choose $X$ as a rigid digital image instead of an irreducible digital image. We express that the digital contractibility implies the reducibility. The next Lemma shows that the converse of this expression is valid. 

\begin{lemma} \label{lemma 2}
	Let $X$ be a digital image with $m \geq 4$ points. 
	
	\textbf{a)} If $C_{m}$ is an empty set, then $X$ is digitally contractible if and only if $X$ is reducible. 
	
	\textbf{b)} If $C_{m}$ is nonempty and $X$ is not digitally homotopy equivalent to $C_{m}$, then $X$ is digitally contractible if and only if $X$ is reducible.
\end{lemma}

\begin{proof}
	\textbf{a)} It is enough to prove that if $X$ is reducible then $X$ is digitally contractible. Let $X$ be a reducible digital image. Then $X$ is digitally homotopy equivalent to an image $X \setminus A$, where $A$ has fewer points than $X$. Let $\ast$ be any point of $X$. If $X$ is digitally homotopy equivalent to the one-point image $\{\ast\}$, then there is nothing to prove. Assume that $X$ is not digitally homotopy equivalent to the one-point image. By Proposition \ref{proposition 2}, we have that $X$ has a simple $m-$loop for any $m \geq 4$. Therefore, there exists a digitally continuous injection $p : C_{m} \rightarrow X$. This is a contradiction because $p$ cannot be an injection. Whereas $X$ has $m$ points, $C_{m}$ is empty for any $m \geq 4$. As a conclusion, $X$ is digitally contractible.
	
	\textbf{b)} Let $X$ be a reducible digital image. Assume that $X$ is not digitally homotopy equivalent to the one point image. Then we have a digitally continuous injection $p : C_{m} \rightarrow X$. The cardinality of $C_{m}$ and $m$ and the cardinality of $X$ is the same. This implies that $p$ is surjective. Therefore, $p$ is a bijection. If we define $q : X \rightarrow C_{m}$ with $q(x) = p^{-1}(x)$, then $q$ is digitally continuous. Indeed, for any $x_{i} \in X$, $i = 1, ..., m$, we find $p^{-1}(x_{i}) \leftrightarrow p^{-1}(x_{i+1})$ because $p^{-1}(x_{i}) = c_{i}$ and $p^{-1}(x_{i+1}) = c_{i+1}$. Hence, we get $p \circ q = id_{X}$ and $q \circ p = id_{X}$. This means that $X$ is digitally homotopy equivalent to $C_{m}$ which is a contradiction. Finally, $X$ is digitally homotopy equivalent to the one point image, i.e. $X$ is digitally contractible. 
\end{proof}

\begin{lemma}\label{lemma 1}
	A digitally connected image $X \subset \mathbb{Z}$ is $2-$contractible if and only if $L_{1}(X) = 1$.
\end{lemma} 

\begin{proof}
	Let $X \subset \mathbb{Z}$ be a $2-$contractible image. Then $X$ is digitally homotopy equivalent to the one-point digital image $\{*\}$. We observe that the one-point is the unique irreducible image in $\mathbb{Z}$. By Theorem \ref{theorem 3}, $L_{1}(X) = L_{1}(\{*\}) = 1$. Conversely, if $L_{1}(X) = 1$, then we have that the number of equivalence classes of $1-$loops is $1$. This means that $X$ is $2-$contractible.
\end{proof}

\quad From the digital image $X$ in Figure \ref{fig1:figure1} (a), we cannot generalize Lemma \ref{lemma 1} in $\mathbb{Z}^{n}$ for $n > 1$. Since $X$ is $8-$connected, $L_{1}(X) = 1$. However, $X$ is not $8-$contractible. The following Corollary is a result of Lemma \ref{lemma 1} and Proposition \ref{proposition 1}.

\begin{corollary}
	Let $X \subset \mathbb{Z}$ be a digitally connected finite image. Then we get TC$(X,2) = 1$.
\end{corollary}

\quad We now provide the digital topological complexity numbers of digital simple closed curves in $\mathbb{Z}^{2}$.

\begin{theorem} \label{theorem 2}
	Let $C_{m}$ be a nonempty $\kappa-$connected digital simple closed curve for any positive integer $m$, where $\kappa \in \{4,8\}$. Then  
	\begin{eqnarray*}\text{TC}(C_{m},\kappa) = 
		\begin{cases}
			1, & m < 5 \\
			2, & m > 5.
		\end{cases}
	\end{eqnarray*}
\end{theorem}

\begin{proof}
	There are two adjacency relations $4$ and $8$ in $\mathbb{Z}^{2}$ so we have two cases. First, consider the $4-$adjacency on $C_{m}$. We catalog the first $12$ nonempty simple closed curves with respect to the number $m$ in this case (see Figure \ref{fig2:figure2}).
	\begin{figure*}[h]
		\centering
		\includegraphics[width=0.80\textwidth]{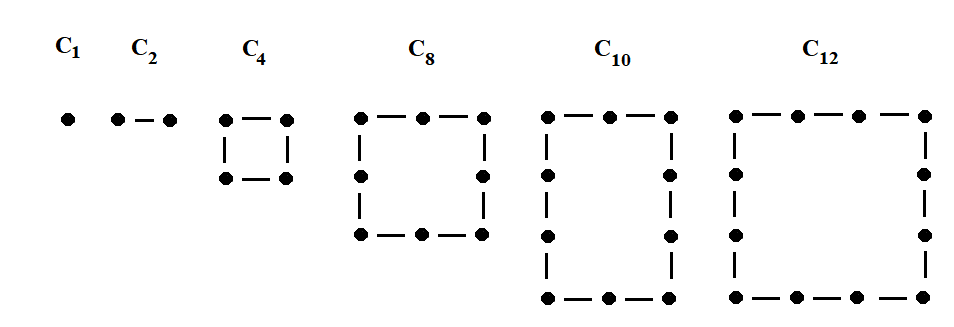}
		\caption{Nonempty simple closed curve $C_{m}$ related to $4-$adjacency for $m = 1, ..., 12$.}
		\label{fig2:figure2}
	\end{figure*} 
We note that some graphics can be different (but homotopy equivalent) in Figure \ref{fig2:figure2}. For instance, the points of $C_{2}$ can be drawn vertically. This does not effect the result as the digital topological complexity number is a homotopy invariant for digital images. For $m>\nolinebreak12$, the list is extended. However, the computation of TC changes only when $m > 5$. Let $m < 5$. We have TC$(C_{m},4) = 1$ because they are $4-$contractible digital images. If $m > 5$, then we show that TC$(C_{m},4) = 2$. Let us choose any two diagonally points (the diagonal can be from left to right or from right to left) on any squares or rectangles for any $m > 5$ and divide the graphic into two parts named as $U_{1}$ and $U_{2}$. Without loss of generality, we assume that $U_{1}$ has one of the diagonal points and $U_{2}$ has the other point. Then \linebreak$U_{1}$ and $U_{2}$ have the same number of points. We set \[V_{1} = \{(x,y) \in C_{m} \times C_{m} \hspace*{0.2cm} | \hspace*{0.2cm} (x,y) \in U_{1}\}\] and \[V_{2} = \{(x,y) \in C_{m} \times C_{m} \hspace*{0.2cm} | \hspace*{0.2cm} (x,y) \in U_{2} \hspace*{0.2cm} \text{or} \hspace*{0.2cm} x \in U_{1}, y \in U_{2} \hspace*{0.2cm} \text{or} \hspace*{0.2cm} x \in U_{2}, y \in U_{1}\}\]
as the subsets of $C_{m} \times C_{m}$. Therefore, we get $C_{m} \times C_{m} = V_{1} \cup V_{2}$. In addition, there exist digitally continuous sections $s_{1} : V_{1} \rightarrow PC_{m}$ and $s_{2} : V_{2} \rightarrow PC_{m}$ of a digital fibration $\pi : PC_{m} \rightarrow C_{m} \times C_{m}$. These satisfy that $\pi \circ s_{1} = id_{V_{1}}$ and $\pi \circ s_{2} = id_{V_{2}}$ and give the desired result for $4-$adjacency. Similarly, we list the first $8$ nonempty simple closed curves with $8-$adjacency in Figure \ref{fig3:figure3}.
\begin{figure*}[h]
	\centering
	\includegraphics[width=0.90\textwidth]{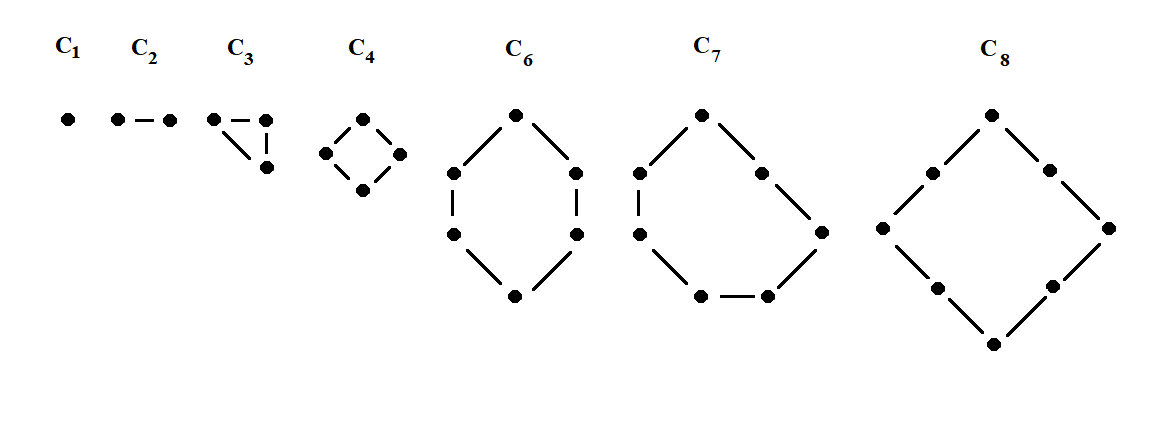}
	\label{fig3:figure3}
	\caption{Nonempty simple closed curve $C_{m}$ related to $8-$adjacency for $m = 1, ..., 8$.}
\end{figure*} 
For $m < 5$, $C_{m}$ is $8-$contractible. Then we have that TC$(C_{m},8) = 1$. For $m >5$, we choose the top and the bottom point of $C_{m}$ (if there are one more top or bottom points, then choose one pair of them such that they are located vertically according to the each other) and divide the graphic into two parts named as $T_{1}$ and $T_{2}$. Without loss of generality, we assume that $T_{1}$ has the bottom point and $T_{2}$ has the top point. We set \[W_{1} = \{(x,y) \in C_{m} \times C_{m} \hspace*{0.2cm} | \hspace*{0.2cm} (x,y) \in T_{1}\}\] and \[W_{2} = \{(x,y) \in C_{m} \times C_{m} \hspace*{0.2cm} | \hspace*{0.2cm} (x,y) \in T_{2} \hspace*{0.2cm} \text{or} \hspace*{0.2cm} x \in T_{1}, y \in T_{2} \hspace*{0.2cm} \text{or} \hspace*{0.2cm} x \in T_{2}, y \in T_{1}\}\]
as the subsets of $C_{m} \times C_{m}$. Then we have digitally continuous sections \linebreak $t_{1} : W_{1} \rightarrow PC_{m}$ and $t_{2} : W_{2} \rightarrow PC_{m}$ of a digital map $ \pi : PC_{m} \rightarrow C_{m} \times C_{m}$ that satisfy that the digital maps $\pi \circ t_{1}$ and $\pi \circ t_{2}$ equal to the identity maps. Moreover, $C_{5}$ is an empty set for both $4$ and $8$ adjacencies. This completes the proof.
\end{proof}

\begin{corollary} \label{corollary 2}
	Let $X\subset \mathbb{Z}^{2}$ be a digitally connected digital image with $m$ points. TC$(X,8) = 1$ for $m < 6$ and TC$(X,4) = 1$ for $m < 8$.
\end{corollary}

\begin{proof}
	Let $m \leq 4$. By Proposition \ref{proposition 3}, $X$ is digitally homotopy equivalent to the one-point digital image. Then, we have that TC$(X,\kappa) = 1$, where $\kappa \in \{4,8\}$. Let $m = 5$. From Proposition \ref{proposition 4} and Proposition \ref{proposition 1}, we get TC$(X,\kappa) = 1$, where $\kappa \in \{4,8\}$. Let $m = 6$ or $7$. Then $C_{m}$ is an empty set with respect to $4-$adjacency. Then $X$ is digitally contractible because $X$ is reducible. This shows that TC$(X,4) = 1$ for $m = 6$ or $m=7$. 
\end{proof}

\quad We are now ready to compute the topological complexity number of any finite digital image in $\mathbb{Z}^{2}$. This characterization indicates that there is no any finite digital image in $\mathbb{Z}^{2}$ whose topopological complexity number is greater than 2.
\begin{corollary}
	Let $X\subset \mathbb{Z}^{2}$ be a $\kappa-$connected digital image with $m$ points. If $C_{m} \neq \emptyset$ and $X$ is digitally homotopy equivalent to $C_{m}$, then we get that 
	\begin{eqnarray*} \text{TC}(X,\kappa) =
		\begin{cases}
			1, & \kappa = 4 \hspace*{0.2cm} \text{and} \hspace*{0.2cm} m < 8 \\
			1, & \kappa = 8 \hspace*{0.2cm} \text{and} \hspace*{0.2cm} m < 6
		\end{cases}
	\end{eqnarray*}
and
\begin{eqnarray*} \text{TC}(X,\kappa) =
	\begin{cases}
		2, & \kappa = 4 \hspace*{0.2cm} \text{and} \hspace*{0.2cm} m \geq 8 \\
		2, & \kappa = 8 \hspace*{0.2cm} \text{and} \hspace*{0.2cm} m \geq 6.
	\end{cases}
\end{eqnarray*}
Otherwise, we have that TC$(X,\kappa) = 1$, where $\kappa \in \{4,8\}$ for any $m$.
\end{corollary}

\begin{proof}
	By Theorem \ref{theorem 2} and Corollary \ref{corollary 2}, it is enough to show that 
	\begin{eqnarray*}
		\text{TC}(X,4) = \text{TC}(X,8) = 1
	\end{eqnarray*} when $C_{m} = \emptyset$ or $X$ is not digitally homotopy equivalent to $C_{m}$. Let $C_{m} = \emptyset$. If $m \leq 4$, then the result holds from Proposition \ref{proposition 3}. If $m \geq 5$, then we have that $X$ is reducible from Proposition \ref{proposition 5}. Hence, the first part a) of Lemma \ref{lemma 2} gives the desired result. Assume that the digital image $X$ is not digitally homotopy equivalent to $C_{m}$. Then TC$(X,\kappa) \neq 2$. Let $C_{m}$ be nonempty and let $m \geq 5$. Since $C_{m}$ is irreducible for $m \geq 5$, $X$ is reducible. Thus, the second part b) of Lemma \ref{lemma 2} completes the proof.
\end{proof}
\section{Digital Higher Topological Complexity of Finite $2$D Digital Images}
\label{sec:2}

\quad We aim to give a general characterization for the digital higher topological complexity computations of any finite digital image especially in $\mathbb{Z}^{2}$ in this section.

\quad We begin with computing the digital higher topological complexity TC$_{n}$ of any one-point digital image for $n \geq 1$. Consider $X = \{\ast\} \subset \mathbb{Z}$ with $2-$adjacency. Let $f \in X^{J_{n}}$ be a constant map at $\ast$. The digital fibration $e_{n} : X^{J_{n}} \rightarrow X^{n}$, defined by $e_{n}(f) = (\ast, \ast, ..., \ast)$, has a digitally continuous map $s : X^{n} \rightarrow X^{J_{n}}$ with $s(\ast, \ast, ..., \ast) = f$ such that $e_{n} \circ s = id$. This shows that TC$_{n}(X) = 1$, where $X$ is a one-point digital image. 

\begin{theorem}
	Let $(X,\kappa)$ be a finite $\kappa-$connected digital image in $\mathbb{Z}$ and $n \geq 1$ be an integer. Then TC$_{n}(X,\kappa) = 1$.
\end{theorem}

\begin{proof}
	If $(X,\kappa)$ is finite and $\kappa-$connected in $\mathbb{Z}$, then it is easy to see that $X$ is $\kappa-$contractible. Hence, it is $\kappa-$homotopy equivalent to the one-point digital image. The digital homotopy invariance of TC$_{n}$ gives the desired result. 
\end{proof}

\quad The digital higher topological complexity computation of a one-point digital image is quite useful because a great majority of digital images in $\mathbb{Z}^{2}$ is digitally contractible (have the same homotopy type with the one-point image). We now examine the digital higher topological complexity of another type which is not homotopy equivalent to the one-point image.

\begin{lemma} \label{Lemma 3}
	Let $(X,\kappa)$ be a $\kappa-$connected digital image. Consider the set \[S_{n}(X) = \{(f, p_{1}, p_{2}, ..., p_{n}) \hspace*{0.2cm} | \hspace*{0.2cm} p_{i} \in Im(f), f \hspace*{0.2cm} \text{is a digital path in $X$}, i = 1, 2, ..., n\}\] in $X^{[0,m]_{\mathbb{Z}}} \times X^{n}$. Then the digital map 
	\begin{eqnarray*}
		&&e_{n}^{'} : S_{n}(X) \longrightarrow X^{n}\\
		&& \hspace*{0.6cm} (f, p_{1}, p_{2}, ..., p_{n}) \longmapsto (p_{1}, p_{2}, ..., p_{n})
	\end{eqnarray*}
is a digital fibrational substitute of the diagonal map $d_{n} : X \rightarrow X^{n}$. 
\end{lemma}  

\begin{remark}
	Note that the adjacency relation on $S_{n}(X)$ is defined as follows: for all $(f, p_{1}, p_{2}, ..., p_{n})$, $(g, q_{1}, q_{2}, ..., q_{n}) \in S_{n}(X)$, $(f, p_{1}, p_{2}, ..., p_{n})$ is $\kappa_{\ast}-$adjacent to $(g, q_{1}, q_{2}, ..., q_{n})$ if $f$ is $\lambda-$adjacent to $g$ and $p_{i}$ is $\kappa-$adjacent to $q_{i}$ for all \linebreak $i = 1, 2, ...,n$, where $\kappa_{\ast}$ is an adjacency relation on $X^{[0,m]_{\mathbb{Z}}} \times X^{n}$ and $\lambda$ is an adjacency relation on digital paths in $X$. 
\end{remark} 

\begin{proof}
	Let $d_{n} : X \rightarrow X^{n}$ be a diagonal map of $X$. Define the digital map \linebreak $h : X \rightarrow S_{n}(X)$ by $h(x) = (\epsilon_{x}, x, x, ..., x)$, where $\epsilon_{x}$ is the digital constant path at $x$. Let $(f, p_{1}, p_{2}, ..., p_{n}) \in S_{n}(X)$. Then there exists $y \in X$ such that $f(0) = y$. Since $X$ is $\kappa-$connected, there exists a digital path $g$ from $x$ to $y$ in $X$, i.e. $g(0) = x$ and $g(1) = f(0) = y$. To show that $h$ is a digital homotopy equivalence, we define a digital map $k : S_{n}(X) \rightarrow X$ with $k(f, p_{1}, p_{2}, ..., p_{n}) = f \ast g(0)$. It is easy to see that $h \circ k$ is digitally homotopic to identity map on $S_{n}(X)$ and $k \circ h$ is digitally homotopic to identity map on $X$. Moreover, we find
	\begin{eqnarray*}
		e_{n}^{'} \circ h(x) = e_{n}^{'}(\epsilon_{x}, x, ..., x) = (x, x, ..., x) = d_{n}(x). 
	\end{eqnarray*}
    Consequently, $e_{n}^{'}$ is a digital fibrational substitute of $d_{n}$.
\end{proof}

\begin{lemma}\label{lemma4}
	TC$_{3}(C_{6},8) = 2$.
\end{lemma}

\begin{proof}
	Let 
	\begin{eqnarray*}
		&&X = C_{6} = \{p_{1} = (0,0), p_{2} = (1,1), p_{3} = (2,1), p_{4} = (3,0),\\ &&\hspace*{1.8cm}p_{5} = (2,-1), p_{6} = (1,-1)\},
	\end{eqnarray*}
where $p_{1} < p_{2} < p_{3} < p_{4} < p_{5} < p_{6}$ (see Figure \ref{fig4:figure4}). Let $e_{3}^{'} : S_{3}(X) \rightarrow X^{3}$ be a digital fibration with $e_{3}^{'}(f, p_{i}, p_{j}, p_{k}) = (p_{i}, p_{j}, p_{k})$ for $i,j,k \in \{1,2,3,4,5,6\}$. We divide $X^{3}$ into two parts. $A_{1}$ consists of triples in $C_{6}$ such that the order of points never changes from left to right, i.e. $p_{i} \leq p_{j} \leq p_{k}$ or if $p_{i} > p_{j}$, then $p_{i} = 6$ and $p_{j} = 1$ (similarly if $p_{j} > p_{k}$, then $p_{j} = 6$ and $p_{k} = 1$). $A_{2}$ consists of elements of $C_{6}$ in which they do not belong to $A_{1}$, i.e. the order of points can change from left to right except using $6$ and $1$ consecutively. Let $(p_{i},p_{j},p_{k}) \in A_{1}$. Using these points, we set a route starting and ending at $p_{i}$ and $p_{k}$, respectively. Then we have a digitally continuous map $s_{1} : A_{1} \rightarrow S_{3}(X)$ with $s_{1}(p_{i},p_{j},p_{k}) = (f, p_{i},p_{j},p_{k})$, where $f$ is the route (digital path from $p_{i}$ to $p_{k}$). It is clear that $e_{3}^{'} \circ s_{1} = id_{S_{3}(X)}$. Similarly, we can construct $s_{2} : A_{2} \rightarrow S_{3}(X)$ with $s_{2}(p_{i},p_{j},p_{k}) = (f, p_{i},p_{j},p_{k})$ over $A_{2}$. Hence, we find that $e_{3}^{'} \circ s_{2} = id_{S_{3}(X)}$. Moreover, we have that $X^{3} = A_{1} \cup A_{2}$. As a result, we get $genus_{\kappa_{\ast}, \lambda_{\ast}}(e_{3}^{'}) = 2$, where $\kappa_{\ast}$ and $\lambda_{\ast}$ are adjacency relations on $S_{3}(X)$ and $X^{3}$, respectively.
\end{proof}
\begin{figure*}[h]
	\centering
	\includegraphics[width=0.70\textwidth]{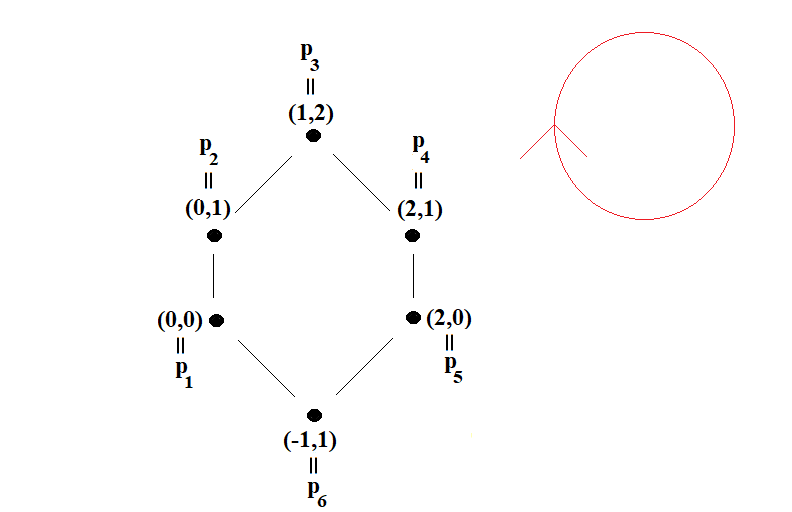}
	\caption{$C_{6}$ with the order of points in it}
	\label{fig4:figure4}
\end{figure*} 
\newpage
\begin{lemma}\label{lemma5}
	TC$_{3}(C_{8},4) = 2$.
\end{lemma}

\begin{proof}
	Let 
	\begin{eqnarray*}
		&&X = C_{8} = \{r_{1} = (0,0), r_{2} = (0,1), r_{3} = (0,2), r_{4} = (1,2), r_{5} = (2,2),\\
		&&\hspace*{1.9cm}r_{6} = (2,1), r_{7} = (2,0), r_{8} = (1,0)\},
	\end{eqnarray*} 
where $r_{1} < r_{2} < r_{3} < r_{4} < r_{5} < r_{6} < r_{7} < r_{8}$ (see Figure \ref{fig5:figure5}). In a similar way of Lemma \ref{lemma4}, we get $B_{1}$ without changing the order of points and $t_{1} : B_{1} \rightarrow S_{3}(X)$ is a digitally continuous map over $C_{1}$ such that $e_{3}^{'} \circ t_{1}$ is identity over $B_{1}$. Changing the order of points in $C_{8}$, we set $B_{2}$ that consists of triples in $C_{8} \times C_{8} \times C_{8}$. The digitally continuous map $t_{2} : B_{2} \rightarrow S_{3}(X)$ gives us $e_{3}^{'} \circ t_{2}$ is identity over $S_{3}(X)$. Hence, we divide $X^{3}$ into two parts $B_{1}$ and $B_{2}$. This proves that TC$_{3}(X,4) = 2$. 
\end{proof}
\begin{figure*}[h]
	\centering
	\includegraphics[width=0.65\textwidth]{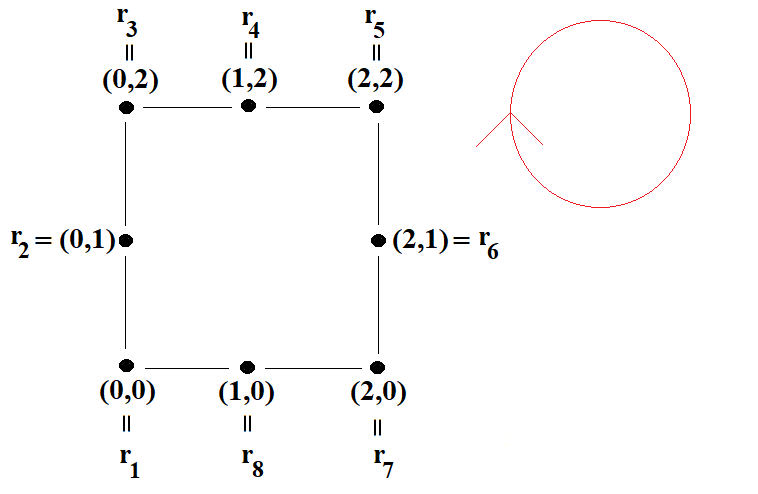}
	\caption{$C_{8}$ with the order of points in it}
	\label{fig5:figure5}
\end{figure*} 

\begin{corollary}\label{lemma6}
	Let $C_{m}$ be a nonempty and $\kappa-$connected digital simple closed curve. Then TC$_{3}(C_{m},4) = 2$ for $m \geq 8$ and TC$_{3}(C_{m},8) = 2$ for $m \geq 6$.
\end{corollary}

\begin{proof}
	The proof is a generalization of Lemma \ref{lemma4} and Lemma \ref{lemma5}. The order of points in $C_{m}$ can be easily constructed for all cases. 
\end{proof}

\quad Corollary \ref{lemma6} can be improved for $n > 3$ and TC$_{n}$ gives the same result with TC$_{3}$ for irreducible digital images:

\begin{theorem}
	Let $C_{m}$ be a nonempty and $\kappa-$connected digital simple closed curve and $n > 2$ be a positive integer. Then 
	\begin{itemize}
		\item TC$_{n}(C_{m},4) = 2$, for $m \geq 8$,
		
		\item TC$_{n}(C_{m},8) = 2$, for $m \geq 6$.
	\end{itemize}	
\end{theorem}

\begin{proof}
	Let $m \geq n$. Let $p_{1}, ..., p_{m}$ be points of $C_{m}$, where $p_{1} < p_{2} < ... < p_{m}$. By using the order, a digital path can be obtained by taking $n$ or less (staying on the same point more than once) of $m$ points. Then the method of Lemma \ref{lemma4} works for this case. Let $m < n$ and $(f, p_{1}, p_{2}, ..., p_{n}) \in S_{n}(X)$. In this case, it is necessary to increase the number of steps of the digital path to be able to have an $n-$step path created with $m$ points. A new $n-$step path is obtained by adding the endpoint of any $m-$step path $f$ to the end of the path $m-n$ times. Since we have $n-$step path, we use its $n$ points in the definition of $S_{n}(X)$. After that, we divide $X^{n}$ into two parts $A_{1}$ and $A_{2}$ again: $n$ points of the digital image in which following the order and not, respectively. Thus, we conclude that the digital Schwarz genus of $e_{n}^{'}$ is $2$.  
\end{proof}

\section{Conclusion}
\label{sec:3}
\quad The aim of this paper is to characterize the digital topological complexity of digitally connected two dimensional finite digital images entirely. We first deal with simple closed curves among digital images because they are irreducible. After giving the results about digital simple closed curves, we examine the topological complexity and the higher topological complexity of all possible digitally connected finite digital images in $\mathbb{Z}$ and $\mathbb{Z}^{2}$.

\quad One of the open problems on this topic is to apply our works on $3-$dimensional digital images. As the number of points that a digital image has in three-dimensional space extremely increases, it is not easy to categorize the topological complexities of these points. Before solving this problem, it is more convenient that trying to categorize the digital images up to digital homotopy equivalence, because of the fact that the topological complexity (and the higher topological complexity) is a homotopy invariant for digital images. Moreover, one can observe the results about the topological complexities of reducible or irreducible images in $\mathbb{Z}^{3}$. This leads us to think more about the characterize digital images up to the digital homotopy equivalence in any dimension of digital topology.

\acknowledgment{The first author is granted as fellowship by the Scientific and Technological Research Council of Turkey TUBITAK-2211-A. In addition, this work was partially supported by Research Fund of the Ege University (Project Number: FDK-2020-21123)}

\end{document}